\documentclass[12pt]{amsart}

\usepackage{amsmath}
\usepackage{amssymb}
\usepackage[all]{xy}

\numberwithin{equation}{section}

\newtheorem{thm}{Theorem}[section]

\newtheorem{lem}[thm]{Lemma}
\newtheorem{prop}[thm]{Proposition}
\newtheorem{cor}[thm]{Corollary}
\newtheorem{rem}[thm]{Remark}
\newtheorem{exam}[thm]{Example}
\newtheorem{exam-nota}[thm]{Example-Notation}
\newtheorem{nota}[thm]{Notation}

\newtheorem{dfn-nota}[thm]{Definition-Notation}

\newtheorem{dfn-lem}[thm]{Lemma-Definition}

\newcommand{\beqa}{\begin{eqnarray*}}
\newcommand{\eeqa}{\end{eqnarray*}}

\newcommand{\fa}{\mbox{${\mathfrak a}$}}

\newcommand{\fk}{\mbox{${\mathfrak k}$}}
\newcommand{\fg}{\mbox{${\mathfrak g}$}}
\newcommand{\fq}{\mbox{${\mathfrak q}$}}
\newcommand{\fl}{\mbox{${\mathfrak l}$}}

\newcommand{\fh}{\mbox{${\mathfrak h}$}}
\newcommand{\fn}{\mbox{${\mathfrak n}$}}

\newcommand{\fp}{\mbox{${\mathfrak p}$}}

\newcommand{\fb}{\mbox{${\mathfrak b}$}}
\newcommand{\fz}{\mbox{${\mathfrak z}$}}
\newcommand{\fm}{\mbox{${\mathfrak m}$}}
\newcommand{\fu}{\mbox{${\mathfrak u}$}}

\newcommand{\C}{\mbox{${\mathbb C}$}}

\newcommand{\Ad}{{\rm Ad}}

\newcommand{\fgl}{\mathfrak{gl}}

\newcommand{\dnone}{{n+1\choose 2}}

\newcommand{\fkn}{\fk_{n+1}}
\newcommand{\fpn}{\fp_{n+1}}

\newcommand{\pin}{\pi_{n}}

\newcommand{\calQ}{\mathcal{Q}}
\newcommand{\nilrad}{\fn_{\calQ}}
\newcommand{\borel}{\fb_{\calQ}}
\newcommand{\orbittower}{X_{\calQ}}
\newcommand{\B}{\mathcal{B}}

\title{K-orbits on the flag variety and strongly regular nilpotent matrices}

\author[M. Colarusso]{Mark Colarusso}
\address{D\'{e}partement de math\'{e}matiques et statistique, Universit\'{e} Laval, Qu\'{e}bec, Canada, G1V 0A6}
\email{mark.colarusso.1@ulaval.ca}

\author[S. Evens]{Sam Evens}
\address{Department of Mathematics, University of Notre Dame, Notre Dame, 46556}
\email{sevens@nd.edu}

\begin{document}
\maketitle

\begin{abstract}
In two 2006 papers, Kostant and Wallach constructed a complexified 
Gelfand-Zeitlin integrable system for the Lie algebra $\fgl(n+1,\C)$
and introduced the strongly regular elements, which are the points
where the Gelfand-Zeitlin flow is Lagrangian.  Later Colarusso studied
the nilfibre, which consists of strongly regular elements such that
each $i\times i$ submatrix in the upper left corner is nilpotent.  In this paper, we prove that every Borel
subalgebra contains strongly regular elements and determine the
Borel subalgebras containing elements of the nilfibre
by using the theory of $K_{i}=GL(i-1,\C) \times GL(1,\C)$-orbits
on the flag variety for $\fgl(i,\C)$ for $2\leq i\leq n+1$. As a consequence, we obtain a more precise
description of the nilfibre. The $K_{i}$-orbits
contributing to the nilfibre are closely related to holomorphic
and anti-holomorphic discrete series for the real Lie groups $U(i,1)$,
with $i \le n$.
\end{abstract}

\section{Introduction}

In a series of papers \cite{KW1, KW2}, Kostant and Wallach study the action of a complex Lie group $A$ on $\fg=\fgl(n+1,\C)$.  The group $A$ is the simply connected, complex Lie group corresponding to the abelian Lie algebra $\fa$ generated by the Hamiltonian vector fields of the Gelfand-Zeitlin collection of functions.  The Gelfand-Zeitlin collection of functions contains $\frac{(n+2)(n+1)}{2}$ Poisson commuting functions and its restriction to each regular adjoint orbit forms an integrable system.  For each function in the collection, the corresponding Hamiltonian vector field is complete.  The action of $A$ on $\fg$ is then defined by integrating the Lie algebra $\fa$. 

Kostant and Wallach consider a Zariski open subset of $\fg$, called
the set of strongly regular elements, which consists of all elements where the
differentials of the Gelfand-Zeitlin functions
 are linearly independent.  The $A$-orbits of strongly regular elements are of dimension $\dnone$ and form Lagrangian submanifolds of regular adjoint orbits.  They coincide with the irreducible components of regular level sets of the moment map for the Gelfand-Zeitlin integrable system.    In \cite{Col1}, the 
first author determined the $A$-orbits of the strongly regular set explicitly.
 
 In this paper, we use the geometry of orbits of a symmetric subgroup
on the flag variety to study the Borel subalgebras containing strongly
regular elements. Let $K_{n+1}:=GL(n,\C)\times GL(1,\C)\subset GL(n+1,\C)$
 be block diagonal matrices and let $\B_{n+1}$ be the flag variety of $\fg$.
Then $K_{n+1}$ acts on $\B_{n+1}$ by conjugation and has finitely many
orbits.  We find a new connection between the theory of $K_{n+1}$-orbits 
on $\B_{n+1}$ and the Gelfand-Zeitlin
integrable system.  In particular, we use the geometry of $K_{n+1}$-orbits on $\B_{n+1}$ to show that 
every Borel subalgebra of $\fg$ contains strongly regular elements.  
We also determine explicitly the Borel subalgebras which contain strongly regular elements in the nilfibre 
of the moment map for the Gelfand-Zeitlin system.  We show that there
are exactly $2^n$ Borel subalgebras containing strongly regular
elements of the nilfibre.  Further, for each of these $2^n$ Borel
subalgebras, the regular nilpotent elements of the Borel is an irreducible
component of the variety of strongly regular elements of the nilfibre, and every 
irreducible component of the variety of strongly regular elements of the nilfibre
arises from a Borel subalgebra in this way.  These $2^n$ Borel subalgebras 
are exactly the Borel subalgebras $\fb$ with the property that 
each projection of $\fb$ to
$\fgl(i,\C)$ for $i=1, \dots, n+1$ is a Borel subalgebra of $\B_i$
whose $K_{i}$-orbit is related via the Beilinson-Bernstein
correspondence to Harish-Chandra modules for the pair $(\fgl(i,\C),
K_{i})$ coming from holomorphic and anti-holomorphic discrete series.
It would be interesting to relate our results to representation
theory, especially to work of Kobayashi \cite{Kobsur}.
For more on the relation between geometry of orbits of a symmetric subgroup
and Harish-Chandra modules, see
\cite{Vg}, \cite{HMSW}, \cite{Collingwood}.

 In more detail, we denote by $x_i$ the upper left $i\times i$ corner of the matrix $x \in \fg$.  The Gelfand-Zeitlin collection of functions is $J_{GZ}=\{f_{i,j}(x), i=1,\dots, n+1,\;j=1,\dots, i\}$, where $f_{i,j}(x)=Tr(x_{i}^{j})$, with
  $Tr(y)$ denoting the trace of $y\in\fgl(i,\C)$.  We denote
by $\fg_i = \{ x_i : x\in \fg \} \cong \fgl(i,\C)$ embedded in $\fg$ as the upper left corner, 
and denote by $G_i \cong GL(i,\C)$
the corresponding algebraic group. The
space $\fa$ spanned by $\{ \xi_f : f\in J_{GZ} \}$ is an abelian Lie algebra.   An
element $x\in \fg$ is called {\it strongly regular} if the set
 $\{ df(x) : f\in J_{GZ} \}$ is
linearly independent in $T_x^*(\fg)$.  Kostant and Wallach showed that the set $\fg_{sreg}$ of
$\fg$ consisting of strongly regular elements is open and Zariski dense.  


The geometry of $\fg_{sreg}$ and its $A$-orbit structure have been studied by the first author in \cite{Col1} and by both authors in \cite{CE}. 
We consider $\Phi:\fg \to \C^{{n+2\choose 2}}$, the moment map for
the Gelfand-Zeitlin integrable system, and for $c\in \C^{{n+2\choose 2}}$,
we let $\Phi^{-1}(c)_{sreg} = \Phi^{-1}(c) \cap \fg_{sreg}$.
 In \cite{Col1}, the first author describes strongly regular $A$-orbits by studying the moment map $\Phi$, and shows that the $A$-orbits in $\Phi^{-1}(c)_{sreg}$ coincide with orbits of an algebraic group action defined on the fibre $\Phi^{-1}(c)_{sreg}$.  In this paper, we develop a completely different approach to studying the $A$-orbits in the nilfibre $\Phi^{-1}(0)_{sreg}:=\Phi^{-1}((0,\dots, 0))_{sreg}$ by finding the Borel subalgebras of $\fg$ which contain elements of $\Phi^{-1}(0)_{sreg}$.  This approach does not require the use of the complicated algebraic group action in \cite{Col1}.    

It is easy to see that an element $x\in\Phi^{-1}(0)_{sreg}$ if and only if $x\in\fg_{sreg}$ and $x_{i}\in\fg_{i}$ is nilpotent for all $i$ (see Remark \ref{rem_git}).  Elements of $\fgl(n+1,\C)$ satisfying the second condition have been studied extensively by Lie theorists and numerical linear algebraists \cite{PS}, \cite{Ov}.  To describe the irreducible components of $\Phi^{-1}(0)_{sreg}$, we study the action of $K_{i}=GL(i-1,\C)\times GL(1,\C)$ on the flag variety $\B_{i}$ of $\fgl(i,\C)$ for $i=1,\dots, n+1$.  If $\fb\subset \fg_{i}$ is a Borel subalgebra, let $\fb_{i-1}=\{x_{i-1}:\, x\in\fb\}$.  Let $Q_{+,i}$ be the $K_{i}$-orbit of the Borel subalgebra of $i\times i$ upper triangular matrices in $\B_{i}$ and let $Q_{-,i}$ be the $K_{i}$-orbit of the Borel subalgebra of $i\times i$ lower triangular matrices in $\B_{i}$.  We show that if $\fb\in Q_{+,i}$ or $\fb\in Q_{-,i}$, then $\fb_{i-1}\subset\fg_{i-1}$ is a Borel subalgebra (see Proposition \ref{prop:borels}).  Let $\calQ=(Q_{1},\dots, Q_{n+1})$ denote a sequence of $K_{i}$-orbits $Q_{i}$ in $\B_{i}$ with $Q_{i}=Q_{+,i}$ or $Q_{-,i}$.  We can then define the closed subvariety of $\B_{n+1}$
$$
\borel=\{\fb\in\B_{n+1}:\, \fb_{i}\in Q_{i},\, i=1,\dots, n+1\}.
$$

\begin{thm}\label{thm:towerintro}(Theorem \ref{thm:areborels})
The subvariety $\borel\subset\B_{n+1}$ is a single Borel subalgebra in $\fg$ that contains the standard Cartan subalgebra of diagonal matrices.  
\end{thm}

Let $\nilrad^{reg}$ denote the regular nilpotent elements of the Borel $\borel$.  

\begin{thm}\label{thm:bigintro}(Theorem \ref{thm:mainthm}). 

 The irreducible component decomposition of the variety $\Phi^{-1}(0)_{sreg}$ is
$$
\Phi^{-1}(0)_{sreg}=\coprod_{\calQ} \nilrad^{reg},
$$
where $\calQ=(Q_{1},\dots, Q_{n+1})$ ranges over all $2^{n}$ sequences
 where $Q_{i}=Q_{+,i}$ or $Q_{-,i}$ for $i>1$, and $Q_{1}=Q_{+,1}=Q_{-,1}$.
The $A$-orbits in $\Phi^{-1}(0)_{sreg}$ are exactly the irreducible
components
$\nilrad^{reg}$, for $\calQ$ as above.
\end{thm}

\begin{rem}\label{rem:intro}
  Theorem 1.2  improves on \cite{Col1}, Theorem 5.2.  In \cite{Col1}, Theorem 5.2, the irreducible components of $\Phi^{-1}(0)_{sreg}$ are described as orbits of the algebraic group $Z_{G_{1}}(e_{1})\times\dots\times Z_{G_{n}}(e_{n})$, where $e_{i}\in\fg_{i}$ is the principal nilpotent Jordan matrix and $Z_{G_{i}}(e_{i})$ is the centralizer of $e_{i}$ in $G_{i}$.  The description of irreducible components
given by $\nilrad^{reg}\cong(\C^{\times})^{n}\times \C^{{n+1\choose 2}-n}$
 is considerably more explicit than the description as orbits (see Examples \ref{ex:Borelexam} and \ref{ex:Borelregexam}).

\end{rem}


To prove Theorem \ref{thm:bigintro}, we use the following characterization of $\Phi^{-1}(0)_{sreg}$ (Theorem \ref{d:sreg} and Remark \ref{rem_git}).
\begin{equation}\label{eq:introcons}
\begin{split}
&x\in\Phi^{-1}(0)_{sreg} \mbox{ if and only if for each } i=2,\dots, n+1:\\
&(1)\; x_{i}\in\fg_{i}, \, x_{i-1}\in\fg_{i-1} \mbox{ are regular nilpotent; and }\\
&(2)\; \fz_{\fg_{i-1}}(x_{i-1})\cap\fz_{\fg_{i}}(x_{i})=0, \end{split}
\end{equation}
where $\fz_{\fg_{i}}(x_{i})$ denotes the centralizer of $x_{i}\in\fg_{i}$.  We fix an $i$ with $2\leq i\leq n+1$ and study the two conditions in (\ref{eq:introcons}).  For concreteness, let us take $i=n+1$ and suppose that $x\in\fg$ satisfies the conditions in (\ref{eq:introcons}).  In particular, $x\in\fg$ is regular nilpotent and therefore contained in a unique Borel subalgebra $\fb_{x}$ of $\fg$.  The conditions in (\ref{eq:introcons}) are $K_{n+1}$-equivariant so that each Borel subalgebra in $K_{n+1}\cdot \fb_{x}$ contains elements satisfying these conditions.  It turns out that very few $K_{n+1}$-orbits contain such Borel subalgebras.  The first condition alone forces the $K_{n+1}$-orbit through $\fb_{x}$ to 
be closed.

\begin{thm}\label{thm:introthm1}(Proposition \ref{prop:closed} and Theorem \ref{prop:nilpsink})
If $x\in\fg$ is regular nilpotent and $x_{n}\in\fg_{n}$ is nilpotent, then $\fb_{x}\in Q$, where $Q$ is closed $K_{n+1}$-orbit. 
\end{thm}  


Theorem \ref{thm:introthm1} allows us to focus our attention on the $(n+1)$ closed $K_{n+1}$-orbits in $\B_{n+1}$.  
\begin{thm}\label{thm:introthm2}(Proposition \ref{prop:centralizers})
If $x$ is a nilpotent element of $\fg$ satisfying condition (2) in (\ref{eq:introcons}) and $x\in\fb\in Q$, with $Q$ a closed $K_{n+1}$-orbit, then $Q=Q_{+,n+1}$ or $Q=Q_{-,n+1}$.  
\end{thm}
\noindent Theorems \ref{thm:introthm1} and \ref{thm:introthm2} force $\fb_{x}\in Q_{+,n+1}$ or $\fb_{x}\in Q_{-,n+1}$.   The proof of Theorem \ref{thm:bigintro} follows by applying this observation successively for each $i$.

It would be interesting to study other strongly regular fibres $\Phi^{-1}(c)_{sreg}$ using the theory of $K_{n+1}$-orbits on the flag variety.  As a step in this direction, we prove:
\begin{thm}\label{thm:introborels}
Every Borel subalgebra $\fb\subset\fg$ contains strongly regular elements.
\end{thm}

The paper is organized as follows. In Section \ref{s:background}, we recall some of the results of \cite{KW1}.  In Section \ref{s:Korbs}, we recall results 
concerning $K_{n+1}$-orbits on the flag variety, and
we prove Theorems \ref{thm:introthm1} and \ref{thm:introthm2}.  In Section \ref{s:components}, we construct the Borel subalgebras $\borel$ and prove Theorem \ref{thm:towerintro}.  At the end of that section, we prove Theorem \ref{thm:bigintro}, which is the main result of the paper.  In the final section, Section \ref{s:sregborels}, we prove Theorem \ref{thm:introborels}.  

We would like to thank Peter Trapa, Leticia Barchini, and Roger Zierau for interesting discussions.  The work
by the 
second author was partially supported by NSA grants H98230-08-0023 and
 H98230-11-1-0151.


\section{Notation and results of Kostant and Wallach}\label{s:background}

Let $\fg=\fgl(n+1,\C)$ be the Lie algebra of $(n+1) \times (n+1)$ complex matrices. If $x\in \fg$, let $x_i$ be
the upper left $i\times i$ corner of $x$, so the $kj$ matrix coefficient
$(x_i)_{kj}$ of $x_i$ is $(x)_{kj}$ 
  if $1 \le k, j \le i$,
and is zero otherwise.  For $i \le n+1$, let $\fg_{i} = \fgl (i,\C) \subset \fg$, regarded as 
the upper left $i\times i$ corner. Let $G_i \cong GL(i,\C)$ be the closed
Lie subgroup of $GL(n+1,\C)$ with Lie algebra $\fg_{i}$.  

Recall the Gelfand-Zeitlin collection of functions $J_{GZ}=\{ f_{i,j}(x): \, i=1,\dots, n+1,\, j=1,\dots, i\},$ where $f_{i,j}(x)=Tr(x_{i}^{j})$.  The collection $J_{GZ}$ is Poisson commutative with respect to the Lie-Poisson structure on $\fg$, and it generates a maximal Poisson commuting subalgebra of $\C[\fg]$ (see \cite{KW1}, Theorem 3.25).  
Let $\fa =\mbox{span} \{ \xi_f : f\in J_{GZ} \}$, where $\xi_f$ is the
Hamiltonian vector field of $f$ on $\fg$.  Then $\fa$ is an abelian Lie algebra, and further $\dim(\fa)=\dnone$ (see \cite{KW1},
Section 3.2).  Let $A$
be the simply connected holomorphic Lie group 
with Lie algebra $\fa$. By Section 3 of \cite{KW1}, the group $A\cong \C^{\dnone}$
integrates the action of $\fa$ on $\fg$. It follows from standard results in
symplectic geometry that $A\cdot x$ is isotropic in
the symplectic leaf $G\cdot x$ in $\fg$.

By definition, $x\in \fg$ is called {\it strongly regular}
if the set $\{ df(x) : f\in J_{GZ} \}$ is linearly independent
in $T^*_x(\fg)$. Let $\fg_{sreg}$ be the (open) set of strongly regular
elements of $\fg$ and recall that a regular element $x$ of $\fg$
is an element whose centralizer
$\fz_{\fg} (x)$ has dimension $n+1$. By a well-known result of Kostant \cite{K},
 if $x \in \fg_{sreg}$, $x_k$ is regular for all $k$ (\cite{KW1}, Proposition
2.6).

We give alternate characterizations of the strongly
regular set in $\fg$. 


\begin{thm}\cite{KW1}
\label{d:sreg}
Let $x\in \fg$. Then the following are equivalent.

\noindent (1) $x$ is strongly regular.

\noindent (2) $x_{i}\in\fg_{i}$ is regular for all $i$, $1\leq i\leq n+1$ and $\fz_{\fg_{i}}(x_{i})\cap \fz_{\fg_{i+1}}(x_{i+1})=0$ for all $1\leq i\leq n$.

\noindent (3) $\dim(A\cdot x)=\dim(A) = \dnone$ and $A\cdot x$ is
Lagrangian in $G\cdot x$.
\end{thm}

By (3) of Theorem \ref{d:sreg} and Theorem 3.36 of \cite{KW1}, 
strongly regular $A$-orbits form the leaves of a polarization of an open subvariety of regular $G\cdot x$.   For this reason, we study the geometry of the $A$-action on $\fg_{sreg}$.  It is useful to consider the moment map for the Gelfand-Zeitlin integrable system, which we refer to as the Kostant-Wallach map.  

\begin{equation}\label{eq:KWmap}
\Phi:\fg\mapsto \C^{{n+2\choose 2}},\; \Phi(x)=(f_{1,1}(x), \dots, f_{n+1,n+1}(x)).
\end{equation}
For $z\in \fg_i$, let $\sigma_i(z)$ equal the collection of
$i$ eigenvalues of $z$ counted with repetitions,
 where here we regard $z$ as an $i\times i$ matrix.

\begin{rem}\label{rem_git}
If $x, y \in \fg$, then $\Phi(x)=\Phi(y)$ if and only if
$\sigma_i(x_i)=\sigma_i(y_i)$ for $i=1, \dots, n+1$.  In particular, $\Phi(x)=(0,\dots, 0)$ if and only if $x_{i}$ is nilpotent for $i=1,\dots, n+1$. 
\end{rem}

It follows from \cite{KW1}, Theorem 2.3 that $\Phi$ is surjective and that for any $c\in\C^{{n+2\choose 2}}$ the variety  $\Phi^{-1}(c)_{sreg}:=\Phi^{-1}(c)\cap\fg_{sreg}$ is nonempty.  By \cite{KW1}, Theorem 3.12, the irreducible components of the variety $ \Phi^{-1}(c)_{sreg}$ coincide with strongly regular $A$-orbits, $A\cdot x$ for $x\in\Phi^{-1}(c)_{sreg}$.  Thus, we can understand the $A$-orbit structure of $\fg_{sreg}$ by studying the geometry of the strongly regular fibres $\Phi^{-1}(c)_{sreg}$.  


\section{The nilfibre and $K$-orbits on the flag variety}\label{s:Korbs}
We study the strongly regular nilfibre $\Phi^{-1}(0)_{sreg}$.  In this section, we find the Borel subalgebras of $\fg_{i}$ which contain elements $x\in\fg_{i}$ satisfying the conditions: 
  \begin{equation}\label{con:theconditions}
\begin{split}
&(1)\; x,\, x_{i-1}\in\fg_{i-1} \mbox{ are regular nilpotent.}\\
&(2)\; \fz_{\fg_{i-1}}(x_{i-1})\cap\fz_{\fg_{i}}(x)=0.\end{split}
\end{equation}
(c.f. Equations (\ref{eq:introcons}).)  In the following sections, we develop an inductive construction that uses these Borel subalgebras to construct Borel subalgebras of $\fg$ that contain elements of $\Phi^{-1}(0)_{sreg}$. 
  
   
\begin{nota} \label{n:thetadef}
We let $\theta$ be the inner automorphism of $\fg_{i}$ given by 
conjugation by the diagonal matrix $c=\mbox{diag}[1,\dots, 1,-1]$, that is, $\theta(x)=cxc^{-1}$.
\end{nota}

   For each $i$, $1\leq i\leq n+1$, let $\mathcal{B}_{i}$ denote the flag variety of $\fg_{i}$.   We let $\fk_{i}$ denote the set of $\theta$ fixed points in $\fg_{i}$ and let $K_{i}$ be the corresponding algebraic subgroup of
$GL(i,\C)$.  Then $\fk_{i}=\fgl(i-1,\C)\oplus\fgl(1,\C)$ consists of block diagonal matrices with one $(i-1)\times (i-1)$ block in the upper left corner and one $1\times 1$ block in the $(i,i)$-entry and the group $K_{i}=GL(i-1,\C)\times GL(1,\C)$ is the set of invertible elements of $\fk_{i}$.  
   
As we observed in the introduction, the set of Borel subalgebras of $\fg_{i}$ containing
elements satisfying (\ref{con:theconditions}) is $K_{i}$-stable.  In this section, we describe the $K_{i}$-orbits on $\mathcal{B}_{i}$ through such Borel subalgebras.  
In Section \ref{s:cond1}, we show that for a Borel subalgebra to contain elements satisfying condition 
(1) in (\ref{con:theconditions}), its $K_{i}$-orbit must be closed.
In Section \ref{s:cond2}, we show that if we also require condition (2), 
the Borel subalgebra must be in
 the $K_{i}$-orbit
through upper or lower triangular matrices.

   
   \begin{nota} \label{nota:standard}
   Throughout the paper, we will make use of the following notation for flags in $\C^{n+1}$.  Let 
   $$
  \mathcal{F}=( V_{0}=\{0\}\subset V_{1}\subset\dots\subset V_{i}\subset\dots\subset V_{n}=\C^{n+1}).
   $$
   be a flag in $\C^{n+1}$, with $\dim V_{i}=i$ and $V_{i}=\mbox{span}\{v_{1},\dots, v_{i}\}$, with each $v_{j}\in\C^{n+1}$.  We will denote the flag $\mathcal{F}$ as follows:
   $$
   v_{1}\subset  v_{2}\subset\dots\subset v_{i}\subset v_{i+1}\subset\dots\subset v_{n+1}. 
   $$
   We denote the standard ordered basis of $\C^{n+1}$ by $\{e_{1},\dots, e_{n+1}\}$. For $1 \le i, j \le n+1$, let $E_{ij}$ be the matrix with $1$
in the $(i,j)$-entry and $0$ elsewhere. 
   
    We also use the following standard notation.  Let $\fh\subset\fg$ be the standard Cartan subalgebra of diagonal matrices.  Let $\epsilon_{i}\in\fh^{*}$ be the linear functional whose action on $\fh$ is given by 
 $$
 \epsilon_{i}(\mbox{diag}[h_{1},\dots, h_{i},\dots, h_{n}])=h_{i}. 
 $$
Let $\alpha_{i}=\epsilon_{i}-\epsilon_{i+1}$, for $1\leq i\leq n$, be the standard simple roots. 
   \end{nota}
   
   \subsection{Basic Facts about $K_{n+1}$-orbits on $\B_{n+1}$}\label{s:basicfacts}
   Our parameterization of $K_{n+1}$-orbits on $\B_{n+1}$ 
follows that of \cite{Yam}, Section 2. 
 For the general case of orbits of a symmetric subgroup on a generalized
flag variety, see
 \cite{RSexp},\cite{Vg}, \cite{Sp}.  
   
   In \cite{Yam}, Yamamoto gives explicit representatives for the 
$K_{n+1}$-orbits on $\B_{n+1}$ in terms of stabilizers of certain flags
 in $\C^{n+1}$.  More precisely,  there are 
$(n+1)$ closed $K_{n+1}$-orbits $Q_{i}=K_{n+1}\cdot \fb_{i}$ on $\B_{n+1}$
where $i=1, \dots, n+1$, and the Borel subalgebra $\fb_{i}$ 
 is the stabilizer of the following flag in $\C^{n+1}$:
   $$
   e_{1}\subset \dots\subset \underbrace{e_{n+1}}_{i}\subset e_{i}\subset \dots\subset e_{n}.
   $$
   Note that if $i=n+1$, then $\fb_{n+1}=\fb_{+}$, the Borel subalgebra of the $(n+1)\times (n+1)$ upper triangular matrices in $\fg$ which stablilizes the
standard flag 
\begin{equation} \label{eq:stdflag}
   {\mathcal{F}}_+ = ( e_{1} \subset \dots \subset e_{n+1}).
\end{equation}
If $i=1$, then $K_{n+1}\cdot \fb_{1}=K_{n+1}\cdot \fb_{-}$, where $\fb_{-}$ is the Borel subalgebra of lower triangular matrices in $\fg$. 
   
   There are $\dnone$ $K_{n+1}$-orbits in $\B_{n+1}$ which are not closed.
  They are of the form $Q_{i,j}=K_{n+1}\cdot \fb_{i,j}$
 for $1\leq i<j\leq n+1$, where $\fb_{i,j}$ 
is the stabilizer of the flag in $\C^{n+1}$:
   $$
  {\mathcal{F}}_{i,j} =   (e_{1}\subset \dots\subset \underbrace{e_{i}+e_{n+1}}_{i}\subset e_{i+1}\subset \dots \subset e_{j-1} \subset \underbrace{e_{i}}_{j}\subset e_{j} \subset \dots \subset e_{n}).
$$
 The unique open $K_{n+1}$-orbit $Q_{1, n+1}$
 on $\B_{n+1}$ is generated by the stabilizer of the flag 
 $$
 e_{1}+e_{n+1}\subset e_{2}\subset\dots \subset e_{n}\subset e_{1}.  
 $$

In order to more easily understand the action of $\theta$ on the
roots of $\fb_{i,j}$, we replace the pair $(\fb_{i,j}, \theta)$
with the equivalent pair $(\fb_+, \theta^\prime)$, where $\theta^\prime$
is an involution of $\fg$ (the pair $\fb_+ \supset \fh$ is a standard pair
with respect to $\theta$, in the language of \cite{RSexp}). To do this,
let  
\begin{equation} \label{eq:vdef} 
v=wu_{\alpha_{i}} \sigma,
\end{equation} 
where $w$ and $\sigma$ are the permutation matrices 
corresponding respectively to the cycles  $(n+1\, n\dots i)$ and $ (i+1\, i+2\dots j)$,
and $u_{\alpha_{i}}$ is the Cayley transform matrix such that
$$
u_{\alpha_{i}}(e_{i})=e_{i} + e_{i+1}, \ 
u_{\alpha_{i}}(e_{i+1})=-e_{i} + e_{i+1}, \
u_{\alpha_{i}}(e_{k})=e_{k}, k\not= i, i+1.
$$
It is easy to verify that $v({\mathcal{F}}_+)={\mathcal{F}}_{i,j}$, and thus
$\Ad(v)(\fb_{+})=\fb_{i,j}$.  We can define a new involution
$\theta^{\prime}:\fg\to \fg$ by
\begin{equation}\label{eq:thetaprimedef}
\theta^{\prime}=\Ad(v^{-1})\circ\theta\circ\Ad(v)=\Ad(v^{-1}\theta(v))\circ\theta.
\end{equation}
By a routine computation
\begin{equation}\label{eq:vij}
v^{-1} \theta(v) = \tau_{i, j} t,
\end{equation}
where $\tau_{i, j}$ is the permutation matrix for the transposition
$(i\, j)$ and $t$ is a diagonal matrix with $t^2$ equal to the
identity.

\begin{rem}\label{r:kprime}
If $K^{\prime}=G^{\theta^{\prime}}$, then $vK^{\prime}v^{-1}=K$,
and if $\fk^{\prime}$ denotes the Lie algebra of $K^{\prime}$, then
$\fk = \Ad(v)(\fk^{\prime})$.
\end{rem}

%
%

Let $\fq$ be the standard parabolic generated by the Borel $\fb_{+}$ and
the  negative simple root spaces 
$\fg_{-\alpha_{i}}, \fg_{-\alpha_{i+1}}, \dots, \fg_{-\alpha_{j-1}}$.
Then $\fq$ has Levi decomposition $\fq = \fl + \fu$, with $\fl$ consisting
of block diagonal matrices of the form
\begin{equation} \label{eq:levifactor}
\fl=\underbrace{\fgl(1,\C)\oplus\dots\oplus \fgl(1,\C)}_{i-1 \mbox{ factors}}\oplus\fgl(j+1-i,\C)\oplus\underbrace{ \fgl(1,\C)\oplus\dots\oplus \fgl(1,\C)}_{n+1-j\mbox{ factors}}.
\end{equation}

\begin{rem} \label{r:thetastablefactors}
Since the permutation matrix corresponding to the transposition
$(i\, j)$ is in the Levi subgroup $L$ with Lie algebra $\fl$,
 the Lie subalgebras
$\fq, \fl$, $\fu$, and $\fgl(j+1-i,\C)$ are $\theta^{\prime}$-stable 
by Equations (\ref{eq:thetaprimedef}) and (\ref{eq:vij}).
\end{rem}

Let $\tilde{\theta}$ denote the restriction of $\theta^{\prime}$
 to $\fgl(j+1-i,\C)$ and let $\tilde{\fk}=\fgl(j+1-i,\C)^{\tilde{\theta}}$
 denote the fixed subalgebra of $\tilde{\theta}$
with corresponding algebraic group
  $\tilde{K}=GL(j+1-i,\C)^{\tilde{\theta}}$.
Let $\fb_{+,j+1-i}$ be the Borel subalgebra of upper triangular matrices
in $\fgl(j+1-i,\C)$, and let $\theta_{j+1-i}$ be the involution of
$\fgl(j+1-i, \C)$ with fixed
set $\fk_{j+1-i}$.

\begin{lem} \label{lem:qijopenlevi}
\par\noindent 1)There is $\tilde{v} \in GL(j+1-i,\C)$ such that 
$\tilde{\theta}=\Ad(\tilde{v}^{-1})\circ \theta_{j+1-i} \circ \Ad(\tilde{v})$.
Further,  
the symmetric subgroup $\tilde{K} \cong K_{j+1-i}$.
\par\noindent 2) $\Ad(\tilde{v})\fb_{+,j+1-i}$
is in the open orbit of $K_{j+1-i}$ on  $\B_{j+1-i}$.
\end{lem}

\begin{proof} 
Let $\tilde{v}=\tilde{w} u_{\alpha_{i}} \sigma$,
 where $\sigma$ and $u_{\alpha_{i}}$ are the same as in the definition
of $v$ from Equation (\ref{eq:vdef}) and $\tilde{w}$ is the 
permutation matrix corresponding to the cycle $(j\, j-1\dots i)$.
Then a routine computation shows that 
$\tilde{\theta}=\Ad(\tilde{v}^{-1})\circ \theta_{j+1-i} \circ \Ad(\tilde{v})$,
and it follows that $\tilde{K} \cong K_{j+1-i}$.
It is easy to show that $\Ad(\tilde{v})\fb_{+,j+1-i}$ is the
stabilizer of the flag
$$
e_{i}+e_{j}\subset e_{i+1}\subset\dots\subset e_{j-1}\subset e_{i},
$$
in $\C^{j+1-i}$, where we take $\{e_{i},\dots, e_{j}\}$ 
as an ordered basis for $\C^{j+1-i}$ (see Notation \ref{nota:standard}).
Hence, the $K_{j+1-i}$-orbit through $\Ad(\tilde{v})\fb_{+,j+1-i}$
is open by the results from \cite{Yam} discussed at the beginning
of this section.
\end{proof}

   \subsection{Analysis of the first condition in (\ref{con:theconditions})}\label{s:cond1}

  Write $\fg=\fkn\oplus\fpn$ where $\fkn$ is the $1$-eigenspace for the involution $\theta$ and $\fpn$ is the $-1$-eigenspace.  Let $\pi_{\fkn}:\fg\to\fkn$ denote the projection of $\fg$ onto $\fkn$ along $\fpn$ and let $\mathcal{N}_{\fkn}$ denote the nilpotent cone in $\fkn$.  For a Borel subalgebra $\fb$ with
nilradical $\fn$, let $\fn^{reg}$ denote the regular nilpotent elements of
$\fn$.  For $\fb=\fb_+$, by \cite{KosWh},
\begin{equation}\label{eq:stdregnil}
\fn_+^{reg} = \{ \sum_{i=1}^{n} a_{i} E_{ii+1}  + y : 
a_{1} \cdots a_{n}\not= 0, y \in [\fn_+,\fn_+] \}.
\end{equation}
  The main result of this section is the following proposition.

\begin{prop}\label{prop:closed}
Suppose that $x\in\fg$ satisfies (1) in  (\ref{con:theconditions}) and suppose that $\fb\subset\fg$ is the unique Borel subalgebra containing $x$.  Then $\fb\in Q$, where $Q$ is a closed $K_{n+1}$-orbit in $\B_{n+1}$.
\end{prop}

Proposition \ref{prop:closed} follows from a stronger result.
\begin{thm}\label{prop:nilpsink}
Let $\fb\subset\fg$ be a Borel subalgebra with nilradical $\fn$
and suppose that the $K_{n+1}$-orbit $Q$ through $\fb$ is not closed.
  Then $\pi_{\fk_{n+1}}(\fn^{reg})\cap\mathcal{N}_{\fk_{n+1}}=\emptyset.$ 
\end{thm}

We first prove Proposition \ref{prop:closed} assuming Theorem \ref{prop:nilpsink}.  

\begin{proof} [Proof of Proposition \ref{prop:closed}]
Suppose $x\in\fg$ satisfies (1) in (\ref{con:theconditions}).  We observe that $\pi_{\fk_{n+1}}(x)=x_{n}$.  Indeed, $\pi_{\fk_{n+1}}(x)$ is the block diagonal matrix $x_{n}\oplus x_{n+1,n+1}$.  But since $x$ satisfies (1) in (\ref{con:theconditions}), we must have:
\begin{equation*}
\begin{split}
0 &=Tr(x)\\
&=Tr(x_{n})+x_{n+1,n+1}\\
&=x_{n+1,n+1},\end{split}
\end{equation*}
where $Tr(x)$ denotes the trace of $x\in\fg$.  Now the proposition follows immediately from Theorem \ref{prop:nilpsink}.
  \end{proof}

\begin{proof}[Proof of Theorem \ref{prop:nilpsink}] 
Let $Q$ be a $K_{n+1}$-orbit in $\B_{n+1}$ which is not closed. By
the results of \cite{Yam} discussed in Section \ref{s:basicfacts}, we
may assume $Q=K_{n+1}\cdot \fb_{i,j}$. By $K_{n+1}$-equivariance of
$\pi_{\fk_{n+1}}$, it suffices to prove the assertion for $\fb=\fb_{i,j}$.
Recall the element $v$ from Equation (\ref{eq:vdef}) such that
$\Ad(v)\fb_{+} = \fb$, and the involution $\theta^{\prime}$, and
let $\pi_{\fk^{\prime}}$ denote projection from $\fg$ to $\fk^{\prime}$
with respect to the $-1$-eigenspace of $\theta^{\prime}$.
By Equation (\ref{eq:thetaprimedef}), it follows that
$\pi_{\fk} = \Ad(v) \circ \pi_{\fk^{\prime}} \circ \Ad(v^{-1})$.
Hence, for $y$ regular in $\fn$, $\pi_{\fk}(y)$ is nilpotent
if and only if $\pi_{\fk^{\prime}}(x)$ is nilpotent, where
$x = \Ad(v^{-1})(y)$ in $\fn_{+}$ is regular nilpotent.

Recall the Levi subalgebra $\fl$ from Equation (\ref{eq:levifactor}),
as well as the parabolic $\fq = \fl + \fu$ and $\fgl(j+1-i,\C)$.
Then $x\in \fq$, so we decompose $x =x_{\fl}+x_{\fu}$, with $x_{\fl}\in\fl$
 and $x_{\fu}\in\fu$. Note that $x_{\fl}$ is regular nilpotent in
  $\fgl(j+1-i,\C)$
by Equation (\ref{eq:stdregnil}). 
Let $\pi_{\fl}:\fq\to \fl$ be the projection of $\fq$ onto $\fl$ along $\fu$.
Since $\fl$ and $\fu$ are $\theta^{\prime}$-stable 
by Remark \ref{r:thetastablefactors}, 
it follows that $\theta^{\prime} \circ \pi_{\fl} = 
\pi_{\fl} \circ \theta^{\prime}$.  Recalling that $\tilde{\theta}$ is the restriction of $\theta^{\prime}$ to $\fgl(j+1-i,\C)$, let $\pi_{\tilde{\fk}}$ be the projection from $\fgl(j+1-i,\C)$ to the fixed point set of $\tilde{\theta}$ along its $-1$-eigenspace.  It follows that 
$$
\pi_{\fl}\circ\pi_{\fk^{\prime}}(x)=\frac{1}{2}(\theta^{\prime}(x_{\fl})+x_{\fl})=\frac{1}{2}(\tilde{\theta}(x_{\fl})+x_{\fl})=\pi_{\tilde{\fk}}\circ\pi_{\fl}(x). 
$$
As is well-known, if an element $y\in\fq$ is nilpotent, then $\pi_{\fl}(y)$ is also nilpotent.  Thus, it suffices to show that $\pi_{\tilde{\fk}}(x_{\fl})$ is not nilpotent.
  
Let $\pi_{\fk_{j+1-i}}$ be the projection
from $\fgl(j+1-i,\C)$ to $\fk_{j+1-i}$ along its $-1$-eigenspace.  By 1) of Lemma \ref{lem:qijopenlevi}, it follows that $\Ad(\tilde{v}) \circ
\pi_{\tilde{\fk}} = \pi_{{\fk}_{j+1-i}} \circ \Ad(\tilde{v})$, and
hence it suffices to prove that $\pi_{\fk_{j+1-i}}(\Ad(\tilde{v})(x_{\fl}))$
is not nilpotent. By 2) of Lemma \ref{lem:qijopenlevi} and equivariance,
it suffices to prove this last assertion in the case where $\fb$ is
the stablizer of the flag
 $$
 e_{1}+e_{n+1}\subset e_{2}\subset\dots \subset e_{n}\subset e_{n+1}.  
 $$
in the open $K_{n+1}$-orbit for the flag variety of $\fgl(n+1,\C)$.

Let $g\in GL(n+1,\C)$ be the matrix
$$
g=\left[\begin{array}{cccc} 1 & & & 0\\
 & \ddots&  & \vdots\\
 & & 1& 0\\
 1& & 0 & 1\end{array}\right]_{\mbox{,}}
 $$
 so that
 \begin{equation}\label{eq:gcong}
 \fb=g\fb_{+}g^{-1}.
\end{equation}
  Let $x\in\fb$ be regular nilpotent.  
Then by Equation (\ref{eq:gcong}), the element $x$ has the form
$$
x=\left[\begin{array}{ccccc}
-a_{1n+1} & a_{12} & a_{13} &\dots & a_{1n+1}\\
-a_{2n+1} & 0 & \ddots& &\\
\vdots &\vdots & \ddots &  a_{n-1n}& \vdots\\
-a_{nn+1}&0& \dots &0 &a_{nn+1}\\
-a_{1n+1}& a_{12} & \dots & \dots & a_{1n+1}\end{array}\right]_{\mbox{,}}
$$
where $a_{ii+1}\neq 0, \, 1\leq i\leq n$.  Then we compute 
$$\pi_{\fkn}(x)=\left[\begin{array}{ccccc}
-a_{1n+1} & a_{12} & a_{13} &\dots &0\\
-a_{2n+1} & 0 & \ddots& &\\
\vdots &\vdots & \ddots & a_{n-1n} & \vdots\\
-a_{nn+1}&0& \dots &0 &0\\
0& 0 & \dots & 0& a_{1n+1}\end{array}\right]
$$
By expanding by cofactors along the $n$th row, the determinant
of the $n\times n$ submatrix of $\pi_{\fkn}(x)$ in the upper left corner is 
$(-1)^n \prod_{i=1}^{n}a_{ii+1}\neq 0$,
which is nonzero by Equation (\ref{eq:stdregnil}).  Hence,
$\pi_{\fkn}(x)$ is not nilpotent, which proves the theorem.
\end{proof}


 \subsection{Analysis of the second condition in (\ref{con:theconditions})}\label{s:cond2}
 We now study the closed $K_{n+1}$-orbits on $\B_{n+1}$ and the condition (2) in (\ref{con:theconditions}). 
 Let $x\in\fg$ be nilpotent and satisfy (2) in (\ref{con:theconditions}) for 
$i=n+1$.  Suppose further that $x\in\fb$, where $\fb$ generates a closed $K_{n+1}$-orbit in $\B_{n+1}$.  In this section, we show that $\fb$ must generate the same orbit in $\B_{n+1}$ as the Borel subalgebra of upper triangular matrices or the Borel subalgebra of lower triangular matrices.  This is an easy consequence of the following proposition.  Recall that $\pi_{n}:\fg\to\fg_{n}$ denotes the projection which sends an $(n+1)\times (n+1)$ matrix to its $n\times n$ submatrix in the upper left corner. 

\begin{prop}\label{prop:centralizers}
Let $\fb\subset\fg$ be a Borel subalgebra
 that generates a closed $K_{n+1}$-orbit $Q$, 
which is neither the orbit of the upper nor the lower triangular matrices.
  Let $\fn=[\fb,\fb]$ and let $\fn_{n}:=\pi_{n}(\fn)$.
  Let $\fz_{\fg}(\fn)$ denote the centralizer of $\fn$ in $\fg$ and 
let $\fz_{\fg_{n}}(\fn_{n})$ denote the centralizer of $\fn_{n}$ in $\fg_{n}$.  Then 
\begin{equation}\label{eq:nilcentralizer}
\fz_{\fg_{n}}(\fn_{n})\cap \fz_{\fg}(\fn)\neq 0.
\end{equation}  
\end{prop}
\begin{proof}

By the $K_{n+1}$-equivariance of the projection $\pi_{n}$, we can assume that $\fb$ is the stabilizer of the flag in $\C^{n+1}$:
$$
e_{1}\subset \dots\subset \underbrace{e_{n+1}}_{i}\subset e_{i}\subset \dots\subset e_{n},
$$
where $1< i< n+1$ (see Section \ref{s:basicfacts}).  The Borel subalgebra $\fb$ is the set of all matrices of the form
\begin{equation}\label{eq:bigmatrix}
\fb=\left[\begin{array}{ccccccccc}
a_{11} &\dots &\dots &a_{1i-1} &\dots &\dots & & a_{1n} & a_{1n+1}\\
0 & \ddots &  & \vdots & & & & \vdots & \vdots\\
\vdots& & & a_{i-1i-1} & &*&  &a_{i-1n} & a_{i-1n+1}\\
& & &0 & a_{ii} & & & \vdots &0\\
\vdots&  & &\vdots & 0 &\ddots & & & \vdots\\
& & & \vdots&\vdots & &\ddots &\vdots & \vdots\\
\vdots & & & 0&0 &\dots&  & a_{nn} & 0\\
0&\dots & \dots & 0& a_{n+1i}& \dots& & a_{n+1n}& a_{n+1n+1}\end{array}\right]_{\mbox{,}}
\end{equation}
$1<i<n+1$.  A system of positive roots for this Borel with respect to the Cartan subalgebra of diagonal matrices is $\Phi=\Phi_{+, n}\cup \Gamma$, where $\Phi_{+,n}$ is the set of positive roots for the strictly upper triangular matrices in $\fgl(n,\C)$ and $\Gamma$ is the following set of roots: 
\begin{equation}\label{eq:roots}
\Gamma=\{-\alpha_{i}-\dots-\alpha_{n}, -\alpha_{i+1}-\dots-\alpha_{n},\dots, -\alpha_{n}\}\cup \{\alpha_{1}+\dots+\alpha_{n}, \alpha_{2}+\dots + \alpha_{n}, \dots, \alpha_{i-1}+\cdots+ \alpha_{n}\}.  
\end{equation} 
Let $\beta=\alpha_{1}+\dots +\alpha_{n-1}$.  Then $\beta\in\Phi$ and $\beta+\alpha\notin\Phi$ for any $\alpha\in \Phi$.  Thus, if $\fg_{\beta}$ is the root space corresponding to $\beta$, $\fg_{\beta}\subset\fz_{\fg}(\fn)$.  Note also that $\fg_{\beta}\subset\fg_{n}$ so that $\fg_{\beta}\subset\fz_{\fg_{n}}(\fn_{n})\cap\fz_{\fg}(\fn)$. 

\end{proof}


\begin{rem}
 It follows from the proof of Proposition \ref{prop:centralizers} that if $x\in\fb$ is nilpotent with $\fb\in Q$, $Q$ a closed $K_{n+1}$-orbit, then $x_{n}\in\fg_{n}$ is nilpotent (in contrast to the situation of
 Theorem \ref{prop:nilpsink}.)   
 \end{rem}

 Let $Q_{+,n+1}=K_{n+1}\cdot \fb_{+}$ denote the $K_{n+1}$-orbit of the Borel subalgebra of upper triangular matrices in $\mathcal{B}_{n+1}$ and let $Q_{-,n+1}=K_{n+1}\cdot \fb_{-}$ denote the $K_{n+1}$-orbit of the Borel subalgebra of the lower triangular matrices in $\mathcal{B}_{n+1}$.

The following result is immediate from
 Propositions \ref{prop:closed} and \ref{prop:centralizers}.

\begin{prop}\label{prop:bothconditions}
Let $x\in\fg$ satisfy both conditions in (\ref{con:theconditions}) and let $\fb$ be the unique Borel subalgebra containing $x$.  Then $\fb\in Q_{+,n+1}$ or $\fb\in Q_{-,n+1}$.  
\end{prop}

We conclude this section with a partial
 converse to Proposition \ref{prop:centralizers}
 that will be useful in proving Theorem \ref{thm:mainthm}.  

\begin{prop}\label{prop:converse}
Let $x\in\fg$ be a regular nilpotent element and 
suppose $x\in\fb\subset Q$, with $Q=Q_{+,n+1}$ or $Q_{-,n+1}$.  Then $x_{n}\in\fg_{n}$ is regular nilpotent and $\fz_{\fg_{n}}(x_{n})\cap \fz_{\fg}(x)=0$.  
\end{prop}  
\begin{proof}
By $K_{n+1}$-equivariance of the projection $\pin:\fg\to\fg_{n}$, we may assume that $\fb=\fb_{+}$ or $\fb_{-}$.  Without loss of generality, we assume $\fb=\fb_{+}$,  since the proof in the case of $\fb=\fb_{-}$ is completely analogous.  Let $\fn=[\fb,\fb]$ and let $\fn^{reg}$ be the regular nilpotent elements of $\fb$.    Then if $x\in\fn^{reg}$, then by Equation
(\ref{eq:stdregnil}),
$x=\lambda_{1} E_{12}+\lambda_{2} E_{23}+\dots+ \lambda_{n} E_{nn+1}+ z$, where $\lambda_{i}\in\C^{\times}$ and $z\in[\fn,\fn]$.  It follows easily that
$x_{n}\in\pi_{n}(\fb)$ is regular nilpotent.

Now we claim that for any $x\in\fn^{reg}$
\begin{equation}\label{eq:centralizers}
\fz_{\fg}(x)\cap\fg_{n}=0.
\end{equation}
It is easily seen that $\fz_{\fg}(x)\cap\fg_{n}=\fz_{\fg_{n}}(x_{n})\cap \fz_{\fg}(x)$. 
Let $x\in\fn^{reg}$, then $x=\Ad(b) e $ for $b\in B$ 
and one knows that $\fz_{\fg}(x)$ has basis 
$\{I, \Ad(b) e, \dots, \Ad(b) e^{n}\}$, 
where $I$ is the $(n+1)\times (n+1)$ identity matrix.  
In matrix coordinates,  
\begin{equation}\label{eq:basis}
\begin{split}
Id&=E_{11}+\cdots+E_{n+1n+1}\\
\Ad(b) e&= X_{1}+c_{1}E_{nn+1}\\
\vdots &\\
\Ad(b) e^{i}&= X_{i}+c_{i}E_{n-i+1 n+1}\\
\vdots & \\
\Ad(b)  e^{n}&=c_{n} E_{1n+1},
\end{split} 
\end{equation}
where $c_{i}\in\C^{\times}$ for $1\leq i\leq n$ and $X_{i}\in\fn$ is of the form $X_{i}=\sum_{j<k, j<n-i+1} \alpha^{i}_{jk} E_{jk}$ with $\alpha^{i}_{jk}\in\C$ for $i\leq n-1$.
Suppose that $y\in\fz_{\fg}(x)\cap \fg_{n}$.  Then there exist scalars $d_{0},d_{1},\dots, d_{n}\in\C$ such that $d_{0}Id+d_{1}\Ad(b) e+\dots + d_{n}\Ad(b) e^{n}\in\fg_{n}$.  But this implies that $d_{0}=0$, by the form of the $X_{i}$ for $1\leq i\leq n-1$ in (\ref{eq:basis}).  Similarly, $d_{1}=0$ by the form of the $X_{i}$ for $i\geq 2$.  Thus, $d_{i}=0$ for all $i$ by induction.  
\end{proof}



  \section{Components of the strongly regular nilfibre}\label{s:components}

Let $1\leq i\leq n+1$.  If $y\in\fg_{i}$ satisfies the conditions in (\ref{con:theconditions}) and $\fb \subset\fg_{i}$ is the unique Borel containing $y$, then Proposition \ref{prop:bothconditions} implies that $\fb\in Q_{+,i}$ or $Q_{-,i}$, where $Q_{+,i}$ is the $K_{i}$-orbit of the upper triangular matrices in $\mathcal{B}_{i}$ and $Q_{-,i}$ is the $K_{i}$-orbit of the lower triangular matrices.  In this section, we develop a construction which links together the $K_{i}$-orbits $Q_{+,i}$ and $Q_{-,i}$ for each $i$, $1\leq i\leq n+1$ to construct Borel subalgebras of $\fg$ that contain elements of $\Phi^{-1}(0)_{sreg}$ and use them to find the irreducible components of $\Phi^{-1}(0)_{sreg}$.  The main tool in our construction is the following simple proposition. 


\begin{prop}\label{prop:borels}
Let $Q$ be a closed $K_{n+1}$-orbit in $\B_{n+1}$ and $\fb\in Q$.  Then $\pi_{n}(\fb)\subset\fg_{n}$ is a Borel subalgebra.  Moreover, $\pi_{n}(\fb)$ is a subalgebra of $\fb$.
\end{prop}
\begin{proof}
By the $K_{n+1}$-equivariance of the projection $\pi_{n}$, it suffices to prove the statement for a representative for the orbit $Q$.  Thus we can assume that $\fb$ is the Borel subalgebra in Equation (\ref{eq:bigmatrix}) where $1\leq i\leq n+1$.  Then $\pi_{n}(\fb)=\fb_{+,n}$, where $\fb_{+,n}$ is the $n\times n$ upper triangular matrices in $\fg_{n}$ and clearly $\fb_{+,n}$ is a subalgebra
of $\fb$.  



\end{proof}


 Suppose we are given a sequence $\mathcal{Q}=(Q_{1},\dots, Q_{n+1})$ with $Q_{i}$ a closed $K_{i}$-orbit in $\mathcal{B}_{i}$.  We call $\calQ$ a sequence of closed $K_{i}$-orbits.  We use the data $\calQ$ and Proposition \ref{prop:borels} to construct a special subvariety $\orbittower$ of $\mathcal{B}_{n+1}$.  For this construction, we view $K_{i}\subset K_{i+1}$ by embedding $K_{i}$ in the upper left corner of $K_{i+1}$.  We also make use of the following notation.  If $\fm\subset\fg$ is a subalgebra, we denote by $\fm_{i}$ the image of $\fm$ under the projection $\pi_{i}:\fg\to\fg_{i}$.  
 
   For $\fb\in Q_{n+1}$, $\fb_{n}$ is a Borel subalgebra by
 Proposition \ref{prop:borels}.  Since $K_{n+1}$ acts transitively
on $\mathcal{B}_{n}$, there is  $k\in K_{n+1}$ 
such that $\Ad(k)\fb_{n}\in Q_{n}$ and the variety 
$$
X_{Q_{n}, Q_{n+1}}:=\{\fb\in \B_{n+1}:\fb\in Q_{n+1},\,  \fb_{n}\in Q_{n}\} 
$$
is nonempty.  Proposition \ref{prop:borels} again implies that $(\Ad(k)\fb_{n})_{n-1}=(\Ad(k)\fb)_{n-1}$ is a Borel subalgebra, so that there exists an $l\in K_{n}$ such that $\Ad(l)(\Ad(k)\fb)_{n-1}\in Q_{n-1}$.  Since $K_{n}\subset K_{n+1}$, the variety 
$$X_{Q_{n-1}, Q_{n}, Q_{n+1}}=\{\fb\in\B_{n+1}: \fb\in Q_{n+1},\,  \fb_{n}\in Q_{n}, \, \fb_{n-1}\in Q_{n-1}\} $$
 is nonempty.  Proceeding in this fashion, we can define a nonempty closed subvariety of $\B_{n+1}$ by 
 \begin{equation}\label{eq:orbittower}
 \orbittower=\{\fb\in \B_{n+1}: \fb_{i}\in Q_{i},\, 1\leq i\leq n+1\}.
 \end{equation}
 \begin{thm}\label{thm:areborels}
 Let $\calQ=(Q_{1},\dots, Q_{n+1})$ be a sequence of closed $K_{i}$-orbits in $\B_{i}$.  Then the variety $\orbittower$ is a single Borel subalgebra of $\fg$ that contains the standard Cartan subalgebra of diagonal matrices.  Moreover, if $\fb\subset\fg$ is a Borel subalgebra which contains the diagonal matrices, then $\fb=\orbittower$ for some sequence of closed $K_{i}$-orbits $\calQ$.  
 \end{thm}
 
 \begin{proof}
 Let $\fh\subset\fg$ be the standard Cartan subalgebra of diagonal matrices.
  We prove that $\orbittower$ is a single Borel subalgebra in $\fg$
 containing $\fh$ by induction on $n$, the case $n=1$ being trivial.  Let $\fb,\, \fb^{\prime}\in\orbittower$.  Then $\fb_{n}, \fb_{n}^{\prime}\in X_{Q_{1},\dots, Q_{n}}$ and by induction $\fb_{n}^{\prime}=\fb_{n}=\fm$, where $\fm\subset\fg_{n}$ is a Borel subalgebra in $\fg_{n}$ containing the standard Cartan $\fh_{n}$.  Since $\fb,\, \fb^{\prime}\in\orbittower$, it follows that $\fb,\, \fb^{\prime}\in Q_{n+1}$, so there exists a $k\in K_{n+1}$ such that $\fb^{\prime}=\Ad(k) \fb$.  Thus, 
 \begin{equation}\label{eq:fm}
 \Ad(k)\fm=\fm.
 \end{equation}
  Suppose $k=\left[\begin{array}{cc} k_{n} & 0\\
 0 & \lambda\end{array}\right]$, with $k_{n}\in GL(n,\C)$ and $\lambda\in\C^{\times}$.  Then Equation (\ref{eq:fm}) implies that $k_{n}\in M$, where $M$ is the Borel subgroup of $GL(n,\C)$ corresponding to $\fm$.  By the second statement of Proposition \ref{prop:borels}, $M\subset B$ from which it follows that $k\in B$, since 
 $$
 k=\left[\begin{array}{cc}
 k_{n}\lambda^{-1} & 0 \\
 0& 1\end{array}\right] \hfill \left[\begin{array} {cc} \lambda I_{n} & 0\\
 0 & \lambda \end{array}\right],
 $$
 where $I_{n}$ denotes the $n\times n$ identity matrix, and the centre of $G$ is contained in all Borel subgroups of $G$.  Thus, $\fb^{\prime} = \Ad(k)\fb=\fb$.
  To see that $\fh\subset\fb$, note that by the induction hypothesis $\fh_{n}\subset\fm\subset\fb$.  But then $\fh=\fh_{n}\oplus \fz(\fg)$, where $\fz(\fg)$ is the centre of $\fg$, so $\fh\subset\fb$.

 
 We now show that every $\fb\subset\fg$ with $\fh\subset\fb$ can be realized as $\orbittower$ for some sequence of closed $K_{i}$-orbits $\calQ=(Q_{1},\dots, Q_{n+1})$.  It is easy to see that if $\calQ$ and $\calQ^{\prime}$ are two different sequences of $K_{i}$-orbits, then $\orbittower\neq X_{\calQ^{\prime}}$.  Since there are exactly $i$ closed $K_{i}$-orbits in $\B_{i}$ (see Section \ref{s:basicfacts}), there are $(n+1)!$ varieties $X_{\calQ}$.  But this is precisely the number of Borel subalgebras in $\fg$ that contain $\fh$.   
 \end{proof}
 
 \begin{nota}
 In light of Theorem \ref{thm:areborels}, we refer to the Borel $\orbittower$ as $\borel$ for the rest of the paper. 
  \end{nota}


To find the Borel subalgebras that contain elements of $\Phi^{-1}(0)_{sreg}$, Proposition \ref{prop:bothconditions} suggests we consider $\borel$, where the sequence $\calQ$ is given by $\calQ=(Q_{1},\dots, Q_{n+1})$ with $Q_{i}=Q_{+,i}$ or $Q_{-,i}$ for all $i$.  Since $Q_{+,1}=Q_{-,1}$, there are $2^{n}$ such Borel subalgebras, and they are precisely the Borel subalgebras identified as the irreducible components of $\overline{\Phi^{-1}(0)_{sreg}}$ in \cite{Col1}, Theorem 5.5.  
\begin{exam}\label{ex:Borelexam}
For $\fg=\fgl(3,\C)$, we have four such Borel subalgebras:
  $$
\begin{array}{cc}
\begin{array}{c}
\mathfrak{b}_{Q_{-},Q_{-}}=\left[\begin{array}{ccc} 
h_{1} & 0 & 0 \\
a_{1}& h_{2} &0\\
a_{2}& a_{3} & h_{3}\end{array}\right]
 \end{array}
&
\begin{array}{c}
\mathfrak{b}_{Q_{+}, Q_{+}}=\left[\begin{array}{ccc} 
h_{1} & a_{1} & a_{2} \\
0 & h_{2} & a_{3}\\
0 & 0 & h_{3}\end{array}\right]
\end{array}
\\
& \\
\begin{array}{c}
\mathfrak{b}_{Q_{+},Q_{-}}=\left[\begin{array}{ccc} 
h_{1}& a_{1} & 0 \\
0 & h_{2} & 0\\
a_{2} & a_{3} & h_{3}\end{array}\right]
\end{array}
&
\begin{array}{c}
\mathfrak{b}_{Q_{-},Q_{+}}=\left[\begin{array}{ccc} 
h_{1} & 0 & a_{1} \\
a_{2} & h_{2}& a_{3}\\
0 & 0 & h_{3}\end{array}\right]
\end{array}
\end{array},
$$
where $a_{i}, h_{i}\in\C$.  
\end{exam}

\begin{thm}\label{thm:mainthm}
If $x\in\Phi^{-1}(0)_{sreg}$, then $x\in\borel$ where $\calQ=(Q_{1},\dots, Q_{n+1})$ is a sequence of closed $K_{i}$-orbits with $Q_{i}=Q_{+,i}$ or $Q_{-,i}$ for all $i$.  Further, let $\nilrad=[\borel,\borel]$ and $\nilrad^{reg}$ denote the regular elements of $\nilrad$.  Then 
\begin{equation}\label{eq:irred}
\Phi^{-1}(0)_{sreg}=\coprod_{\calQ} \nilrad^{reg}
\end{equation}
is the irreducible component decomposition of the variety $\Phi^{-1}(0)_{sreg}$, where $\calQ$ runs over all $2^{n}$ sequences of closed $K_{i}$-orbits $(Q_{1},\dots, Q_{n+1})$ with $Q_{i}=Q_{+,i}$ or $Q_{-,i}$.  
\end{thm} 

\begin{proof}

We prove the first statement by induction on $n$, the case $n=1$ being trivial.  Let $x\in\Phi^{-1}(0)_{sreg}$.  Then $x$ is regular nilpotent and therefore contained in a unique Borel subalgebra $\fb\subset\fg$.  Further, since $x\in\fg_{sreg}$, part (2) of Theorem 2.1 implies that $x$ satisfies the conditions of Equation (3.1) for $i=2,\dots, n+1$, so by Proposition 3.10, $\fb\in Q_{+,n+1}$ or $\fb\in Q_{-,n+1}$.  It also follows immediately from part (2) of Theorem 2.1 and Remark 2.2 that $x_{n}\in \Phi_{n}^{-1}(0)_{sreg}$, where $\Phi_{n}:\fgl(n,\C)\to \C^{\dnone}$ is the Kostant-Wallach map for $\fgl(n,\C)$, that is, $\Phi_{n}(y)=(f_{1,1}(y)\dots, f_{n,n}(y))$ for $y\in\fgl(n,\C)$.  By induction, $x_{n}\in\fb_{\calQ_{n}}\subset\fg_{n}$, where $\calQ_{n}=(Q_{1},\dots, Q_{n})$ is a sequence of closed $K_{i}$-orbits with $Q_{i}=Q_{+,i}$ or $Q_{-,i}$ for $i=1,\dots, n$. Since $x_{n}$ is regular nilpotent and $\fb_{n}$ is a Borel subalgebra by Proposition 4.1, it follows that $\fb_{n}=\fb_{\calQ_{n}}$.  It then follows from definitions that $\fb=\fb_{\calQ}$ where $\calQ=(Q_{1},\dots, Q_{n}, Q_{n+1})$ and $Q_{n+1}= Q_{+,n+1}$ or $Q_{-,n+1}$.  

Let $\calQ=(Q_{1},\dots, Q_{n+1})$ be a sequence of closed $K_{i}$-orbits with $Q_{i}=Q_{+,i}$ or $Q_{-,i}$.  To prove Equation (\ref{eq:irred}), we first observe that $\nilrad^{reg}\subset\Phi^{-1}(0)_{sreg}$.  Indeed, let $x\in\nilrad^{reg}$, then $x\in\borel$ is regular nilpotent and $\borel\in Q_{+,n+1}$ or $Q_{-,n+1}$.  Proposition \ref{prop:converse} then implies that $x_{n}$ is regular nilpotent and $\fz_{\fg_{n}}(x_{n})\cap \fz_{\fg}(x)=0$.  But then $x_{n}\in\fn_{\calQ_{n}}^{reg}$, with $\calQ_{n}=(Q_{1},\dots, Q_{n})$.  By induction we conclude that $x_{n}\in \Phi_{n}^{-1}(0)_{sreg}$.  By part (2) of Theorem \ref{d:sreg}, it follows that $x$ is strongly regular and hence $\nilrad^{reg}\subset\Phi^{-1}(0)_{sreg}$.  

 We now show that $\nilrad^{reg}$ is an irreducible component of $\Phi^{-1}(0)_{sreg}$.  Observe that $\nilrad^{reg}$ is an irreducible variety of dimension $\dnone$.  By \cite{KW1}, Theorem 3.12,  $\Phi^{-1}(0)_{sreg}$ is a variety of pure dimension $\dnone$ whose irreducible and connected components coincide. Thus, if 
 $$
 \coprod_{i=1}^{k} (\Phi^{-1}(0)_{sreg})(i)
 $$
 is the irreducible component decomposition of $\Phi^{-1}(0)_{sreg}$, then $\nilrad^{reg}\subset (\Phi^{-1}(0)_{sreg})(j)$ for some $j$ and $\nilrad^{reg}$ is open in $(\Phi^{-1}(0)_{sreg})(j)$.  From the first statement of the theorem, it follows that 
\begin{equation*}
(\Phi^{-1}(0)_{sreg})(j)=\coprod_{\calQ^{\prime}}\fn_{\calQ^{\prime}}^{reg},
\end{equation*}
where the disjoint union is taken over a subset of the set of all sequences $(Q_{1}^{\prime},\dots, Q_{n+1}^{\prime})$ with $Q^{\prime}_{i}=Q_{+,i}$ or $Q_{-,i}$.  Thus $\nilrad^{reg}$ is both open and closed in $(\Phi^{-1}(0)_{sreg})(j)$ forcing $(\Phi^{-1}(0)_{sreg})(j)=\nilrad^{reg}$, since $(\Phi^{-1}(0)_{sreg})(j)$ is connected.  
\end{proof}

Theorem \ref{thm:mainthm} provides an alternative proof of the following
corollary from \cite{Col1} (see Theorem 5.2).
  Recall that the group $A\cong\C^{\dnone}$ is obtained by integrating the Lie algebra of Hamiltonian vector fields $\fa=\mbox{span}\{\xi_{f_{i,j}}: 1\leq i\leq n,\, 1\leq j\leq i\}$ on $\fg$ (see Section \ref{s:background}).  
\begin{cor}\label{c:littlecor}
There are $2^{n}$ $A$-orbits in $\Phi^{-1}(0)_{sreg}$. 
\end{cor}
\begin{proof}
By \cite{KW1}, Theorem 3.12 the $A$-orbits in $\Phi^{-1}(0)_{sreg}$ coincide with the irreducible components of $\Phi^{-1}(0)_{sreg}$.  The result then follows immediately from Theorem \ref{thm:mainthm}.  
\end{proof}

\begin{exam}\label{ex:Borelregexam}
For $\fg=\fgl(3,\C)$, Theorem 4.5 implies that the four $A$-orbits in $\Phi^{-1}(0)_{sreg}$ are the regular nilpotent elements of the Borel subalgebras given in Example 4.4.  
  $$
\begin{array}{cc}
\begin{array}{c}
\mathfrak{n}_{Q_{-},Q_{-}}^{reg}=\left[\begin{array}{ccc} 
0 & 0 & 0 \\
a_{1}& 0 &0\\
a_{3}& a_{2} & 0\end{array}\right]
 \end{array}
&
\begin{array}{c}
\mathfrak{n}_{Q_{+}, Q_{+}}^{reg}=\left[\begin{array}{ccc} 
0 & a_{1} & a_{3} \\
0 & 0 & a_{2}\\
0 & 0 & 0\end{array}\right]
\end{array}
\\
& \\
\begin{array}{c}
\mathfrak{n}_{Q_{+},Q_{-}}^{reg}=\left[\begin{array}{ccc} 
0& a_{1} & 0 \\
0 & 0& 0\\
a_{2} & a_{3} & 0\end{array}\right]
\end{array}
&
\begin{array}{c}
\mathfrak{n}_{Q_{-},Q_{+}}^{reg}=\left[\begin{array}{ccc} 
0 & 0 & a_{1} \\
a_{2} & 0& a_{3}\\
0 & 0 & 0\end{array}\right]
\end{array}
\end{array},
$$
where $a_{1}, \, a_{2}\in\C^{\times}$ and $a_{3}\in\C$.  
\end{exam}

\begin{rem}
It follows from Theorem \ref{thm:mainthm} that the irreducible components of the variety $\overline{\Phi^{-1}(0)_{sreg}}$ are precisely the nilradicals $\nilrad$ of the Borel subalgebras $\borel$.  This result was proved earlier in \cite{Col1}, Theorem 5.5, but we regard the argument in this paper as more conceptual and the results more extensive.  The connection between the strongly regular elements $\Phi^{-1}(0)_{sreg}$ and the regular nilpotent elements of the Borel subalgebras $\borel$ was not understood in \cite{Col1} (see Remark \ref{rem:intro}). 

\end{rem}

\section{Borel subalgebras and strongly regular elements}\label{s:sregborels}
In this section, we show that every Borel subalgebra $\fb\subset\fg$ contains strongly regular elements.  We first consider the Borel subalgebras $\fb\subset\fg$ which contain the standard Cartan subalgebra $\fh$ of diagonal matrices in $\fg$.  We refer to such Borel subalgebras as \emph{standard} Borel subalgebras.

Recall the involution $\theta$ from Notation \ref{n:thetadef}, so
$K_{n+1} = G^{\theta}$.  
Let $H$ be the Cartan subgroup of $G$ with Lie algebra $\fh$.  Let $W=N_{G}(H)/H$ denote the Weyl group of $G$ with respect to $H$
and consider its subgroup $W_{K_{n+1}} = N_{K_{n+1}}(H)/H$, the Weyl group
of $K_{n+1}$.

\begin{lem}\label{l:kweyl}
Let $\fb, \, \fb^{\prime}$ be standard Borel subalgebras of $\fg$ which generate the same closed $K_{n+1}$-orbit in $\mathcal{B}_{n+1}$.  Then $\fb^{\prime}=w\cdot\fb$ for some $w\in W_{K_{n+1}}$.
\end{lem}

\begin{proof}
We have $\fb^{\prime}=k\cdot\fb$ for some $k\in K_{n+1}$, and we let 
$B^{\prime}$ be the Borel subgroup with Lie algebra $\fb^{\prime}$.
 Then the Cartan subalgebra $k\cdot \fh\subset\fb^{\prime}$ is $\theta$-stable.  Since any two $\theta$-stable subalgebras of $\fb^{\prime}$ are conjugate by an element of $B^{\prime}\cap K_{n+1}$ (see for example Proposition 1.2.1 \cite{RSexp}), there exists $b\in B^{\prime}\cap K_{n+1}$ such that $bk\cdot \fh=\fh$.  But then $bk\in N_{K_{n+1}}(H)$ and $bk\cdot \fb=\fb^{\prime}$ and the result follows.  
\end{proof}

\begin{prop}\label{p:standard}
Let $\fb\subset\fg$ be a standard Borel subalgebra.  Then $\fb$ contains strongly regular elements.  
\end{prop}

\begin{proof}
Let $\fb\subset \fg$ be a standard Borel subalgebra.  We construct an element $x\in\fb$ satisfying the following conditions.  
  \begin{equation}\label{eq:three}
\begin{split}
&(1)\; x_{i}\in\fg_{i} \mbox{ is regular semisimple for } i=1,\dots, n+1.\\
&(2)\; \fz_{\fg_{i}}(x_{i-1})\cap \fz_{\fg_{i}}(x_{i})=0 \mbox{ for } i=2,\dots,n+1.
\end{split}
\end{equation}
 The construction of $x$ proceeds by induction on $n$; 
the case $n=1$ being trivial.  
By Theorem \ref{thm:areborels}, the Borel subalgebra $\fb=\borel$ 
for some sequence of closed $K_{i}$-orbits $\calQ=(Q_{1},\dots, Q_{n}, Q_{n+1})$.
  Thus, $\fb_{n}=\fb_{Q_{1},\dots, Q_{n}}$ is a standard Borel subalgebra of $\fg_{n}$.  By induction, there exists $y\in \fb_{Q_{1},\dots, Q_{n}}$
 satisfying the conditions in (\ref{eq:three}) for $1\leq i\leq n$. 
 Since $y\in \fb_{Q_{1},\dots, Q_{n}}$ is regular semisimple, 
there exists $b\in B_{Q_{1},\dots, Q_{n}}$ such that $\Ad(b)y=h\in \fh_{n}$, 
where $\fh_{n}$ is the standard Cartan subalgebra of diagonal matrices 
in $\fg_{n}$ 
and $B_{Q_{1},\dots, Q_{n}}\subset G_{n}$ is the Borel subgroup 
corresponding to the Borel subalgebra $\fb_{Q_{1},\dots, Q_{n}}$.

By the results discussed in Section \ref{s:basicfacts}, 
the orbit $Q_{n+1}= K_{n+1}\cdot \fm$, 
where $\fm$ is the stabilizer of the flag 
$$
e_{1}\subset\dots\subset\underbrace{e_{n+1}}_{i}\subset\dots\subset e_{n}
$$
 for some $i=1,\dots, n+1$.  
The Borel subalgebra $\fm$ is given explicitly in Equation (\ref{eq:bigmatrix}).  Since $\borel$ and $\fm$ are standard Borel subalgebras contained in $Q_{n+1}$, Lemma \ref{l:kweyl} implies that $w\cdot \fm=\borel$ for some $w\in W_{K_{n+1}}$.  Recall the subset $\Gamma$ of positive roots for $\fm$ defined in Equation (\ref{eq:roots}) and note that the corresponding root spaces are
in the far right column and bottom row. 
  Now we define an element $z\in\fm$ as follows.  
Let $z_{n}=\Ad(w^{-1}) h$, and let the coefficient of $z$ in any root space $\fg_{\alpha}$ with $\alpha\in \Gamma$ be $1$, and define $z_{n+1 n+1}$ so that $z_{n+1 n+1}\neq h_{ii}$ for any $i$.  A simple matrix calculation shows that
the eigenvalues of $z$ are $\{ h_{11}, \dots, h_{nn}, z_{n+1 n+1}\}$, so
that $z$ is regular semisimple.  Since $\fz_{\fg_n}(z_n) = \fh_n$ and
$\fz_{\fh_n}(z)=0$, it follows that $\fz_{\fg_n}(z_n) \cap \fz_{\fg}(z)=0$.

Now consider $x=\Ad(b^{-1}w) z$.  Since $b^{-1}w\in K_{n+1}$, it follows that $x$
satisfies the conditions in (\ref{eq:three}) for $i=n+1$.  By construction $x_n = \Ad(b^{-1}w)z_n = y$.  
It then follows from Theorem \ref{d:sreg} that $x\in\fg_{sreg}$ and 
that $x$ satisfies the conditions of (\ref{eq:three}) for all $i=1,\dots, n+1$.  
Further, since $b\in B_{Q_{1},\dots, Q_{n}}\subset B_{Q_{1},\dots, Q_{n+1}}$ by Proposition \ref{prop:borels}, the strongly regular element $x\in\borel$.  
\end{proof}

Using Proposition \ref{p:standard}, we can prove the main result of the section.

\begin{thm}\label{thm:sregborels}
Let $\fb\subset\fg$ be any Borel subalgebra.  Then $\fb$ contains strongly regular elements of $\fg$.
\end{thm}

\begin{proof}
For ease of notation, we denote the 
flag variety $\mathcal{B}_{n+1}$ of $\fgl(n+1,\C)$ by $\B$.  
Define $$\B_{sreg}=\{\fb\in\B\; |\; \fb\cap\fg_{sreg}\neq\emptyset\}.$$
  By Proposition \ref{p:standard}, $\B_{sreg}$ is nonempty and 
we claim that $\B_{sreg}$ is open in $\B$.  
To see this, we use the Grothendieck resolution 
$\tilde{\fg}=\{ (x,\fb)\in \fg \times \B \; | \; x\in\fb\}$,
as well as the morphisms  
$\mu:\tilde{\fg}\to \fg,\; \mu(x,\fb)=x$, and $\pi: \tilde{\fg}\to\B,\; \pi(x,\fb)=\fb$.  
Then $\pi$ is a smooth morphism of relative dimension $\dim\fb={n+2\choose 2}$ by \cite{CG}, Section 3.1 and \cite{Ha}, Proposition III.10.4. Thus,
$\pi$ is a flat morphism (\cite{Ha}, Theorem III.10.2) and hence
open by Exercise III.9.1 in \cite{Ha}.
  Now consider $\mu^{-1}(\fg_{sreg})=\{(x,\fb): x\in\fg_{sreg},\, x\in\fb\}$.  Since $\fg_{sreg}\subset\fg$ is open, $\mu^{-1}(\fg_{sreg})\subset\tilde{\fg}$ is open, 
and it follows that $\pi(\mu^{-1}(\fg_{sreg}))$ is open in $\B$. 
 But it is easily seen that $\pi(\mu^{-1}(\fg_{sreg}))=\B_{sreg}$, which proves the claim.  

We now show that the closed set $Y=\B\setminus \B_{sreg}$ is empty.
 Suppose to the contrary that $\fb\in Y$.  
  It follows from Theorem \ref{d:sreg} that $H$ acts on $\fg_{sreg}$ by conjugation and hence on $\B_{sreg}$ and $Y$ by conjugation.  Thus, $H\cdot \fb\subset Y$.  Since $Y$ is closed, we have $\overline{H\cdot \fb}\subset Y$.  Now $\overline{H\cdot \fb}$ contains a closed $H$-orbit, 
and the closed $H$-orbits in $\B$ are precisely the standard Borel subalgebras of $\fg$ (\cite{CG}, Lemma 3.1.10). Hence, there is a standard Borel $\fb_{std}$
in $Y$, which contradicts Proposition \ref{p:standard}.

\end{proof}

\bibliographystyle{amsalpha.bst}

\bibliography{bibliography}

\end{document}